\documentclass[review]{elsarticle}

\usepackage{lineno,hyperref}
\modulolinenumbers[5]

\journal{Constructive Approximation}
\usepackage{}

\input epsf
\usepackage{amssymb}
\usepackage{amsfonts}
\usepackage{amsmath}
\usepackage{mathrsfs}
\usepackage{amsthm}

\numberwithin{equation}{section}
\newtheorem{theorem}{Theorem}[section]
\newtheorem{lemma}{Lemma}[section]
\newtheorem{de}{Definition}[section]
\newtheorem{proposition}{Proposition}[section]

\newtheorem{corollary}{Corollary}[section]

\newtheorem{remark}{Remark}[section]

\begin{document}

\begin{frontmatter}

\title{Asymptotic properties of biorthogonal
polynomials systems related to Hermite and Laguerre
polynomials}\tnotetext[mytitlenote]{This work is supported by the
Natural Science Foundation of China (11301060), China Postdoctoral
Science Foundation (2013M541234, 2014T70258) and Outstanding
Scientific Innovation Talents Program of DUFE (DUFE2014R20). }

\author{Yan Xu}
\address{School of Mathematics and Quantitative Economics, Center
for Econometric analysis and Forecasting, Dongbei University of
Finance and Economics, Liaoning, 116025, PR China}
\cortext[myauthor]{Email: yan\underline{\hbox to
1.5mm{}}xu@dufe.edu.cn}

\begin{abstract}
In this paper, the structures to a family of biorthogonal
polynomials that approximate to the Hermite and Generalized Laguerre
polynomials are discussed respectively. Therefore, the asymptotic
relation between several orthogonal polynomials and combinatorial
polynomials are derived from the systems, which in turn verify the
Askey scheme of hypergeometric orthogonal polynomials. As the
applications of these properties, the asymptotic representations of
the generalized ~Buchholz, ~Laguerre, ~Ultraspherical
  (Gegenbauer), Bernoulli, Euler, Meixner and Meixner-Pllaczekare polynomials are derived
  from the theorems directly. The relationship between Bernoulli and Euler polynomials
are shown as a special case of the characterization theorem of the
Appell sequence generated by $\alpha$ scaling functions.
\end{abstract}


\end{frontmatter}

\linenumbers

\section{Introduction}

The Hermite polynomials follow from the generating function
\begin{eqnarray}
e^{xz-\frac{z^2}{2}}=\sum_{m=0}^{\infty}\frac{H_m(x)}{m!}z^m, &&z\in
\mathbb{C}, x\in\mathbb{R}
\end{eqnarray}
which gives the Cauchy-type integral

\begin{eqnarray}
H_m(x)=\frac{m!}{2i\pi }\oint e^{xz-\frac{z^2}{2}}z^{-(m+1)}dz.
\end{eqnarray}
The derivatives of the Gaussian function,
$G(x)=\frac{1}{\sqrt{2\pi}}e^{-x^2/2}$, produce the Hermite
polynomials by the relation, $ (-1)^mG^{(m)}(x)=H_m(x)G(x),
m=0,1,\ldots.$ Therefore the orthonormal property of the Hermite
polynomials,

\begin{equation*}
\frac{1}{m!}\int_{-\infty}^{\infty}H_m(x)H_n(x)G(x)dx=\delta_{m,n},
\end{equation*}
can be considered as a biorthogonal relation between the derivatives
of the Gaussian function, $ \{(-1)^nG^{(n)}:n=0,1,\ldots\}$ and the
Hermite polynomials, $\{ \frac{H_m}{m!}:m=0,1,\ldots\}.$

The Hermite polynomials have been extensively studied since the
pioneer article of C. Hermite \cite{Hermite1864} in 1864. It has
many interesting properties and applications in several branches of
mathematics, physical sciences and engineering. A rich source of
orthogonal polynomials, for instance, Gegenbauer \cite{Temme1989},
Laguerre \cite{lagurre1}, Charlier\cite{lagurre2}, Jacobi
\cite{L_T_her}, Meixner-Pollaczek, Meixner, Krawtchouk and Hahn-type
polynomials \cite{Ferreira_Meixner} have asymptotic approximations
in terms of Hermite polynomials, which is known as the famouse Askey
scheme\cite{Temme_L_Askey01, Koekoek_Askey98, T_L2000}.  The
asymptotic relations among other hypergeometric orthogonal
polynomials and its q-analogue can be found in \cite{Ferreira,
Koekoek_Askey98}. The asymptotic representations of other families
of polynomials, such as the generalized Bernoulli, Euler, Bessel and
Buchholz polynomials are also considered\cite{Dominici,
L_T_Bernoulli, Wong_Bessel}.

In \cite{Lee}, S. L. Lee extended the biorthogonal properties
between the derivatives of the Gaussian function and Hermite
polynomials to a family of scaling functions with compact support
and a family of Appell sequences which approximate to the Gaussian
function and Hermite polynomials respectively. The Appell
polynomials are also called scaling biorthogonal polynomials which
are eigenfunctions of a linear operator and the distributional
derivatives of $\phi$ are the eigenfunctions of its adjoint
corresponding to the same eigenvalues \cite{05AMS}. In particular,
the Appell polynomials generated by the uniform \textit{B-spline}
are the classical Bernoulli polynomials which asymptotic approximate
to the Hermite polynomials by suitably normalized\cite{Lee}.

The main objectives of this paper is to extend these properties to a
family of non-scaling functions that approximate to the generating
functions and to construct a family of biorthogonal polynomials that
approximate to the Hermite polynomials and Laguerre polynomials
respectively. The asymptotic properties between Hermite, Laguerre
and other orthogonal polynomials which are known as Askey Scheme are
derived from these theorems as simple cases. The relationship
between the Appell sequence polynomials and the $\alpha$-scaling
compact support functions are also considered.

This paper is organized as follows: In section 2, we present the
framework of the biorthogonal polynomials system which related to
Hermite polynomials and Gaussian function. The characterization
theorem of the Appell sequence generated by the scaling functions
are also shown in this section. In section 3, as the applications of
senction 2, the asymptotic relations among Bernoulli polynomials,
Euler polynomials and B-splines are studied. The new identical
relation between Bernoulli and Euler polynomials are shown as a
special case. In section 4, we generalize the asymptotic
relationship to a family of hypergeometric orthogonal polynomials
related to Hermite polynomials. The asymptotic properties of
Generalized Buchholz polynomials and Ultraspherical (Gegenbauer)
polynomials are considered as the applications. In section 5, we
generalize the biorthogonal systems to the Laguerre polynomials and
derive several asymptotic properties in Askey scheme.

\section{Biorthogonal polynomials approximate to Hermite polynomial}

Let $C^\infty(\mathbb{R})$ denoted for the space of infinitely
differentiable functions. If $\phi:
C^\infty(\mathbb{R})\rightarrow\mathbb{R}$ is a linear functional,
we shall write $\langle \phi,\nu\rangle=\phi(\nu), \nu\in C^\infty
(\mathbb{R})$. The linear functional $\phi$ is continuous if and
only if there is a compact subset $K$ of $\mathbb{R}$, a constant $C
> 0$ and an integer $k\geq 0$ such that
\begin{equation*}
|\langle\phi,\nu\rangle|\leq C \max_{j\leq k}\sup_{x\in K}|\nu^{(j)}(x)|.
\end{equation*}
We denote the space of distributions with compact support by $\mathcal {E}^\prime(\mathbb{R}) $. Integrable
functions and measures with compact supports belong to $\mathcal {E}^\prime(\mathbb{R}) $. If $f$ is a compactly
supported integrable function then it is associated with the distribution, which we
still denote by $f$, defined by
\begin{eqnarray*}
\left\langle f, \nu\right\rangle:=\int_\mathbb{R}\nu(x)f(x)dx, && \nu\in C^\infty(\mathbb{R}).
\end{eqnarray*}
If $m$ is a compactly supported measure on $\mathbb{R}$, then it is
associated with the distribution, which we still denote by $m$,
defined by
\begin{eqnarray*}
\left\langle m,\nu\right\rangle:=\int_\mathbb{R}\nu(x)d m(x), && \nu \in C^\infty(\mathbb{R}).
\end{eqnarray*}
Any $\phi \in \mathcal {E}^\prime (\mathbb{R})$ has derivatives $\phi^{(n)}$ of any order n and they are defined by
\begin{eqnarray*}
\left\langle \phi^{(n)}, \nu\right\rangle=(-1)^n\left\langle \phi,
\nu^{(n)}\right\rangle, && n=0,1,\dots.
\end{eqnarray*}
Taking a compactly supported distribution $\phi\in\mathcal
{E}^\prime(\mathbb{R})$, then for any integer $n\geq0$,
\begin{equation*}
\left\langle \phi^{(n)},e^{(\cdot)z}\right\rangle=(-1)^n\left\langle
\phi,z^ne^{(\cdot)z} \right\rangle=(-1)^nz^n\widehat{\phi}(iz),
\end{equation*}
where $\widehat{\phi}(\cdot)$ denote the Fourier transform of
$\phi(x)$.

If $\widehat{\phi}(0)\neq0$,
\begin{equation}\label{z^n}
\left\langle (-1)^n\phi^{(n)}, \frac{e^{(\cdot)z}}{\widehat{\phi}(iz)}\right\rangle=z^n
\end{equation}
in a neighborhood of $0$. Since $\phi$ is compactly supported, $\widehat{\phi}$ is analytic. So we can
define a sequence of polynomials, $P_m$, by the generating function
\begin{equation}\label{pm}
\frac{e^{xz}}{\widehat{\phi}(iz)}=\sum_{m=0}^\infty \frac{P_m(x)}{m!}z^m.
\end{equation}
It follows from (\ref{z^n}) and (\ref{pm}) that for any integer $n\geq 0$,
\begin{equation}
z^n=\sum_{m=0}^\infty \left\langle(-1)^n\phi^{(n)},\frac{P_m(x)}{m!}\right\rangle z^m,
\end{equation}
which gives the biorthogonal relation
\begin{equation}\label{orth}
\left\langle(-1)^n\phi^{(n)},\frac{P_m(x)}{m!}\right\rangle=\delta_{m,n}.
\end{equation}

\begin{de}
A sequence of polynomials, $\{P_m(x): m \in N\}$, is an Appell
sequence if $P_m(x)$ is a polynomial of degree $m$ and
$$ P'_m(x)=mP_{m-1}.$$
\end{de}

Differentiating (\ref{pm}) with respect to $x$ and equating coefficients of $z^m$ in the resulting equation gives
\begin{eqnarray}
P'_m(x)=mP_{m-1}(x), && m=1,2,\ldots,
\end{eqnarray}
which implies $P_m(x)$ are Appell sequence of polynomials. Therefore
the distribution $\phi$ generates an Appell sequence of polynomials
by the generating function $\frac{e^{xz}}{\widehat{\phi}(iz)}$.

We consider a family of sequences of biorthogonal polynomials,
$\{P_{N,m}:m=0,1,\ldots, N=1.2\ldots\}$, that are generated by a
sequence of functions, $\phi_{N}$, which converges to the Gaussian
function
\begin{eqnarray}
\frac{e^{xz}}{\hat{\phi}_N(iz)}=\sum_{m=0}^\infty
\frac{P_{N,m}(x)}{m!}z^m.
\end{eqnarray}

\begin{de}
 Let $ \tilde{\phi}_N$ be the standardized form of $\phi_N$,
 \begin{equation*}
 \tilde{\phi}_N(x)=\sigma_N \phi_N(\sigma_N x+\mu_N),
 \end{equation*}
where $\mu_N$ and $\sigma^2_N$ are the mean and variance of
$\phi_N$. Define the standardized form of the biorthonormal
polynomials, $\tilde{P}_{N,m}$, of $\{P_{N,m}: m = 0, 1, \ldots\}$
by
\begin{equation}
\tilde{P}_{N,m}(x) =\sigma_{N}^{-m} P_{N,m}(\sigma_N x + \mu_N).
\end{equation}
Then the following biorthogonal relations for the standardized
biorthogonal polynomials follow from (\ref{orth}):
\begin{equation*}
 \left\langle  (-1)^n \tilde{\phi}_N^{(n)},
 \tilde{P}_{N,m}\right\rangle=\delta_{m,n},  \forall
 m,n \geq 0.
 \end{equation*}
 \end{de}
Further, the generating functions of $ \tilde{P}_{N,m} $ are given
by\cite{Lee}
\begin{equation}\label{sorth}
\frac{e^{xz}}{\widehat{\tilde{\phi}}_N(iz)}=\sum_{m=0}^\infty
\frac{\tilde{P}_{N,m}(x)}{m!}z^m.
\end{equation}

\begin{theorem}\label{Th2}
Let $\tilde{\phi}_N(x)$ satisfy the following conditions:

(1) There exist constant $r>0$ such that for any $\varepsilon$,
there is a sufficient large $N_0$, for any $N>N_0$, it holds
\begin{eqnarray}\label{abs1}
\left| \widehat{\tilde{\phi}}_N(iz)-e^{\frac{z^2}{2}}\right|\leq \varepsilon, && |z|<r.
\end{eqnarray}

(2) Let $\{\tilde{P}_{N,m}(x): m=0,1,\ldots\} $ be the biorthogonal
polynomials generated by the functions, $\tilde{\phi}_N$, as in
(\ref{sorth}).

Then for each $m=0,1,\ldots$, $\tilde{P}_{N,m}(x)$ converges locally uniformly to the Hermit polynomial $H_m(x)$ as $N$ goes to infinity.
\end{theorem}
\begin{proof}

Since $\widehat{\tilde{\phi}}_N(0)=1$, we can choose a neighborhood
$U$ of the origin so that
$\left|\widehat{\tilde{\phi}}(iz)\right|\geq\frac{1}{2}$ and $\left|
e^{\frac{z^2}{2}}\right|\geq \frac{1}{2}$ for all $z\in U$. Take a
circle $C$ inside $U$ with center at $0$ and radius $r$ so that
(\ref{abs1}) is satisfied.

Noting
\begin{equation}\label{taylor}
\frac{ e^{xz}}{\widehat{\tilde{\phi}}_N(iz)}-\frac{e^{xz}}{e^{\frac{z^2}{2}}}=\sum_{m=0}^\infty\frac{\tilde{P}_{N,m}(x)-H_m(x)}{m!}z^m.
\end{equation}

The coefficients of the Taylor series (\ref{taylor}) are represented
by the Cauchy's integral formula
\begin{equation*}
\tilde{P}_{N,m}(x)-H_m(x)=\frac{m!}{2\pi i}\oint_C
\frac{e^{xz}(e^{\frac{z^2}{2}}-\widehat{\tilde{\phi}}_N(iz))}{z^{m+1}\widehat{\tilde{\phi}}_N(iz)e^{\frac{z^2}{2}}}dz.
\end{equation*}
Noting  $\tilde{\phi}_N(x)$ satisfy the condition (1), which means
$\exists$ $ r>0$, for a real number $A>0$, there is a sufficient
large $N_0$, for any $N> N_0$, it holds
\begin{eqnarray}
 \left|
\widehat{\tilde{\phi}}_N(iz)-e^{\frac{z^2}{2}}\right|\leq\frac{A}{\sigma_N},
&& |z|<r.
\end{eqnarray}

 Therefore
\begin{eqnarray*}
\left|\tilde{P}_{N,m}(x)-H_m(x)\right|&&\leq\frac{m!}{2\pi }
\oint_C\frac{\left|e^{xz}\right|\left|e^{\frac{z^2}{2}}-\widehat{\tilde{\phi}}_N(iz)
\right|}{r^{m+1}\left|\widehat{\tilde{\phi}}_N(iz)\right|
\left|
e^{\frac{z^2}{2}}\right|}|dz|\\
&&\leq\frac{m!}{2\pi }\oint_C\frac{e^{xRe(z)}\frac{A}{\sigma_N}}{r^{m+1}\left|\widehat{\tilde{\phi}}_N(iz)\right|
\left|
e^{\frac{z^2}{2}}\right|}|dz|\\
&&\leq\frac{4(m!)e^{rx}A}{\sigma_N r^m}
\end{eqnarray*}
Since $\sigma_N\rightarrow\infty$ as $N\rightarrow \infty$, It follows that for each $m$, $\tilde{P}_{N,m(x)}\rightarrow H_m(x)$ uniformly on compact sets.
\end{proof}

\subsection{Appell sequence generated by scaling functions}
The refinement equation
\begin{eqnarray}\label{scale phi}
\phi_n(x)=\int_\mathbb{R} \alpha \phi_n(\alpha x -y)dm_n(y), && x\in
\mathbb{R},  n=1,2,\ldots,
\end{eqnarray}
where $\alpha > 1 $ and $\{m_n\}$ is a sequence of probability
measures with finite first and second moments. Equivalently,
(\ref{scale phi}) can be expressed in term of Fourier transforms in
the frequency domain in the form
\begin{eqnarray}\label{FT}
\widehat{\phi}_n(\mu)=\widehat{m}_n\left(\frac{\mu}{\alpha}\right)\widehat{\phi}_n\left( \frac{\mu}{\alpha}\right), && \mu \in \mathbb{R}.
\end{eqnarray}
\begin{theorem}\label{2.2}
If two Appell sequence polynomials, $P_m(x)$ and $Q_m(x)$ are generalized by
\begin{equation}\label{mask}
\frac{e^{xz}}{\widehat{\psi}(iz)}=\sum_{m=0}^{\infty}
P_m(x)\frac{z^m}{m!},
\end{equation}
and
\begin{equation}\label{phi}
\frac{e^{xz}}{\widehat{\phi}(iz)}=\sum_{m=0}^{\infty}
Q_m(x)\frac{z^m}{m!}
\end{equation}
respectively. Then  $\phi(x)$ is a $\alpha$-scaling compact
supported function with mask $\psi(x)$ if and only if
\begin{equation}\label{scale}
\sum_{k=0}^{m}\alpha^{-m}\binom {m}{k}P_k\left(\alpha x\right)Q_{m-k}\left({\alpha x}\right)=Q_m(2x).
\end{equation}

\end{theorem}
\begin{proof}
Since
\begin{equation*}
\frac{e^{\frac{xz}{\alpha}}}{\widehat{\psi}(\frac{iz}{\alpha})}=\sum_{m=0}^{\infty}
\alpha^{-m}P_m (x)\frac{z^m}{m!},\space\,\space\
\frac{e^{\frac{xz}{\alpha}}}{\widehat{\phi}(\frac{iz}{\alpha})}=\sum_{m=0}^{\infty}
\alpha^{-m}Q_m(x)\frac{z^m}{m!},
\end{equation*}
then
\begin{eqnarray}\label{1}
\frac{e^{\frac{2xz}{\alpha}}}{\widehat{\psi}(\frac{iz}{\alpha})
\widehat{\phi}(\frac{iz}{\alpha})}&=&\left(\sum_{j=0}^{\infty}
\alpha^{-m}P_m(x)\frac{z^m}{m!}\right)\left(\sum_{m=0}^{\infty} \alpha^{-m}Q_m(x)\frac{z^m}{m!}\right)\\
&=&\sum_{m=0}^\infty\left(\sum_{k=0}^m
\alpha^{-m}\binom{m}{k}P_k(x)Q_{m-k}(x)\right)\frac{z^m}{m!}.
\end{eqnarray}
Suppose that the Appell sequence polynomials, $P_m(x)$ and $Q_m(x)$,
satisfy (\ref{scale}),
\begin{equation*}
\sum_{k=0}^{m}\alpha^{-m}\binom {m}{k}P_k\left(\alpha x\right)Q_{m-k}\left({\alpha x}\right)=Q_m(2x).
\end{equation*}
Then
\begin{equation}\label{scale2}
\sum_{m=0}^\infty\left(\sum_{k=0}^{m}\alpha^{-m}\binom{m}{k}P_k\left( \frac{\alpha x}{2}\right)Q_{m-k}\left( \frac{\alpha x}{2}\right)\right)\frac{z^m}{m!}=\sum_{m=0}^\infty Q_m(x)\frac{z^m}{m!}.
\end{equation}
Therefore
\begin{eqnarray}\label{1}
\frac{e^{xz}}{\widehat{\psi}(\frac{iz}{\alpha})\widehat{\phi}(\frac{iz}{\alpha})}
&=&\sum_{m=0}^\infty\left(\sum_{k=0}^{m}\alpha^{-m}\binom{m}{k}P_k\left( \frac{\alpha x}{2}\right)Q_{m-k}\left( \frac{\alpha x}{2}\right)\right)\frac{z^m}{m!}\\
&=&\sum_{m=0}^\infty Q_m(x)\frac{z^m}{m!}=\frac{e^{xz}}{\widehat{\phi}(iz)}.
\end{eqnarray}
We see that
\begin{eqnarray*}
\widehat{\psi}(\frac{iz}{\alpha})\widehat{\phi}(\frac{iz}{\alpha})
=\widehat{\phi}(iz),
\end{eqnarray*}
which imply $\phi(x)$ is scaling function with mask $\psi(x)$.

Suppose that $\phi(x)$ is a $\alpha$-scaling compact supported function with mask $\psi(x)$, satisfying the scaling equation (\ref{scale phi}). Then the Fourier transform $\widehat{\phi}$ is given in (\ref{FT}) shows that
\begin{equation}
\widehat{\phi}\left(\frac{iz}{\alpha}\right)\widehat{\phi}\left(\frac{iz}{\alpha}\right)=\widehat{\phi}(iz).
\end{equation}
It follows that
\begin{eqnarray*}\label{1}
\sum_{m=0}^\infty
Q_m(2x)\frac{z^m}{m!}&=&\frac{e^{2xz}}{\widehat{\phi}(iz)}
=\frac{e^{2xz}}{\widehat{\psi}(\frac{iz}{\alpha})\widehat{\phi}(\frac{iz}{\alpha})}\\
&=&\left(\sum_{m=0}^{\infty} \alpha^{-m}P_m\left(\alpha
x\right)\frac{z^m}{m!}\right)\left(\sum_{m=0}^{\infty}
 \alpha^{-m}Q_m\left( \alpha x\right)\frac{z^m}{m!}\right)\\
&=&\sum_{m=0}^\infty\left(\sum_{k=0}^{m}\alpha^{-m}\binom{m}{k}P_k\left(
\alpha x\right)Q_{m-k}\left(\alpha x\right)\right)\frac{z^m}{m!}.
\end{eqnarray*}
Therefore
\begin{equation*}
\sum_{k=0}^{m}\alpha^{-m}\binom{m}{k}P_k\left(\alpha x\right)Q_{m-k}\left({\alpha x}\right)=Q_m(2x).
\end{equation*}
\end{proof}

\section{Generalized Bernoulli polynomials, Euler Polynomials and
B-splines}

{\it  B-splines with order $N$}, which is denoted as $B_N(\cdot)$,
is defined by the induction as
$$
B_1(x)\,\,=\,\,\begin{cases}
1 & \text{if} \>\> x\in[0,1),\\
0 & \text {otherwise}
\end{cases}
$$
and for $N\geq 1$
$$
B_N\,\,=\,\, B_1*B_{N-1},
$$
where $*$ denotes the operation of convolution which is defined by
$$
( f\ast g)(t) := \int_{-\infty}^{+\infty} f (t -y)g(y)dy,
$$
for $f $ and $g$ in $L^2(\mathbb{R})$.

The Fourier transform of $B_N(x)$ is
\begin{eqnarray*}
\widehat{B}_N(\omega)=\left(
\frac{1-e^{-i\omega}}{i\omega}\right)^N, && \omega\in \mathbb{R}.
\end{eqnarray*}
$B_N(x)$ also satisfies the scaling function as follow:
\begin{equation}\label{scale B}
B_N(x)=2\sum_{j=0}^N \frac{1}{2^N}\binom {N}{j}B_N(2x-j),
\end{equation}
where the mask $\psi_N(k):=\frac{1}{2^n}\binom{n}{k}$. Equivalently,
(\ref{scale B}) can be expressed in term of Fourier transforms in
the frequency domain in the form:
\begin{eqnarray*}
\widehat{B}_N(\omega)&= &\widehat{\psi}_N
(\omega/2)\widehat{B}_N(\omega/2),
\end{eqnarray*}
where the Fourier transform of the mask is $\widehat{\psi}_N
(\omega)=\left( \frac{1+e^{-i\omega}}{2}\right)^N .$

The asymptotic properties of B-splines have a long history going
back to the physicist Arnold Sommerfeld who showed that Gaussian
function can be approximated by B-splines point-wise in 1904
\cite{somer}. In 1992, Unser and his colleagues \cite{Unser} proved
that the sequence of normalized and scaled B-splines tends to
Gaussian function in $L^p$ space as the order $N$ increases. L. H.
Y. Chen, T. N. T. Goodman and S. L. Lee\cite{Goodman} considered the
convergence orders of scaling functions which asymptotic to
normality. A result due to Ralph Brinks \cite{ralph} generalized
Unser's result to the derivatives of the B-splines. Yan Xu and R. H.
Wang \cite{YW} gave the convergence orders of the approximation
processes and showed the asymptotic relationship
 among B-splines, Eulerian numbers and Hermite polynomials.

\begin{theorem}\cite{YW}\label{limtB} Let be $k\in
\mathbb{N}$, for $N> k+2$, the sequence of the $k$-th derivatives,
$B^{(k)}_{N}$, of the $B$-spline converges to the $k$-th derivative
of the Gaussian function
\begin{equation}
\left(\frac{N}{12}\right)^\frac{k+1}{2}B^{(k)}_{N}\left(\sqrt{\frac{N}{12}}x+\frac{N}{2}\right)
=\frac{1}{\sqrt{2\pi}} D^k\exp\left({-\frac{x^2}{2}}\right)
+O\left(\frac{1}{N}\right),
\end{equation}
and
\begin{equation}
\lim_{d\rightarrow\infty}\left\{\left(\frac{N}{12}\right)^\frac{k+1}{2}B^{(k)}_{N}
\left(\sqrt{\frac{N}{12}}x+\frac{N}{2}\right)\right\}
=\frac{(-1)^k}{\sqrt{2\pi}}H_k(x)G(x),
\end{equation}
where the limit may be taken point-wise or in $L^p(\mathbb{R}), p\in
[2,\infty)$.
\end{theorem}

 Generalized
Bernoulli\cite{Lundell, Weinmann, Srivastava} and Euler
polynomials\cite{Todorov85, Todorov93} of degree $m$, order $N$ and
complex argument $z$, denoted respectively by $B_m^N(z)$ and
$E_m^N(z)$ can be defined by their generating functions,
\begin{eqnarray}\label{defE}
\frac{\omega^N e^{\omega z}}{(e^\omega-1)^N}=\sum_{m=0}^\infty \frac{B_m^N(z)}{m!}\omega^n,&   & |\omega|<2\pi,\\
\frac{2^N e^{\omega z}}{(e^\omega+1)^N}=\sum_{m=0}^\infty \frac{E_m^N(z)}{m!}\omega^n,&   & |\omega|<\pi.\label{defB}
\end{eqnarray}

In paper \cite{Lee}, S. L. Lee has proved that the Appell
polynomials generated by the uniform \textit{B-spline} of order $N$
are the generalized Bernoulli polynomials of order $N$ and when
suitably normalized they converge to the Hermit polynomials as
$N\rightarrow \infty$. Since the B-splines approximate the Gaussian
function\cite{Goodman, Unser,ralph, Lee, YW}, they can also be used
as a filter in place of the Gaussian filter for linear
scale-space\cite{wl}.
\begin{corollary}\cite{Lee}
\begin{equation*}
\lim_{N\rightarrow\infty}\left( \frac{12}{N}\right)^{\frac{m}{2}}B_m^{N}\left(\sqrt{\frac{N}{12}}z+\frac{N}{2}\right)=H_m(z).
\end{equation*}
\end{corollary}
\begin{proof}
Recall that the uniform B-spline, $B_N(x)$, of order $N$, is the
scaling function satisfying
\begin{equation*}
B_N(x)=2\sum_{j=0}^N \frac{1}{2^N}\binom {N}{j}B_N(2x-j),
\end{equation*}
where the mask $\psi_N(k):=\frac{1}{2^n}\binom{n}{k}$. Equivalently,
the scaling equation can be expressed in term of Fourier transforms
in the frequency domain in the form:
\begin{eqnarray*}
\widehat{B}_N(\omega)&= &\widehat{\psi}_N (\omega/2)\widehat{B}_N(\omega/2),
\end{eqnarray*}
where the Fourier transform of the mask is $\widehat{\psi}_N (\omega)=\left( \frac{1+e^{-i\omega}}{2}\right)^N .$
From the fourier transform of $B_N$ we have
\begin{equation}
\widehat{B}_N(i\omega)=\left( \frac{e^\omega-1}{\omega}\right)^N.
\end{equation}

By (\ref{pm}), (\ref{orth}) and the generating function of
$B_m^N(z)$, the generalized Bernoulli polynomials, $\{B_m^N(z): m=0,
1,\ldots\}, $ are biorthogonal to the derivatives of the
\textit{B-splines}, $B_N^{(n)}$,
\begin{equation}
\left\langle (-1)^n B_N^{(n)}(z), \frac{B_m^N(z)}{m!}\right\rangle=\delta _{m,n}.
\end{equation}

The standardized \textit{B-splines},
\begin{equation*}
\tilde{B}_N(x)=\sqrt {\frac{N}{12}}B_N\left( \frac{N}{12}x +\frac{N}{12}\right),
\end{equation*}
converges uniformly to the Gaussian function, $G(x)$, and an
estimate of the rate of convergence is given in \cite{Goodman, Lee,
YW} .

By Theorem \ref{Th2}, we have
\begin{equation*}
\lim_{N\rightarrow\infty}\left( \frac{12}{N}\right)^{\frac{m}{2}}B_m^{N}\left(\sqrt{\frac{N}{12}}z+\frac{N}{2}\right)=H_m(z).
\end{equation*}
\end{proof}

It is well known that the binomial distributions converge to the normal distribution in the sense that
\begin{equation}
\lim_{N\rightarrow\infty}\sum_{k=0}^{[x_N]}\frac{1}{2^{N}}\binom {N}{k}=\frac{1}{\sqrt{2\pi}}\int_{-\infty}^x e^{-t^2/2}dt,
\end{equation}
where $x_N= \sqrt{N}x/2+N/2$.
Let $\sigma_N=\sqrt{N}x/2$ and $\mu_N=N/2$, then the standardized binomial distributions, $\tilde{\psi}_N(x):=\sigma_N\psi_N(\sigma_Nx+\mu_N)$,
converges uniformly to the Gaussian function $G(x)=\frac{1}{\sqrt{2\pi}}e^{-x^2/2}$.

The Appell polynomials generated by the binomial distributions are
the generalized Euler polynomials of order $N$ and when suitably
normalized they converge to the Hermit polynomials as
$N\rightarrow\infty$.
\begin{corollary}
\begin{equation}
\lim_{N\rightarrow\infty}\left(
\frac{4}{N}\right)^{\frac{m}{2}}E_m^N(\frac{\sqrt{N}}{2}z+\frac{N}{2})=H_m(z).
\end{equation}
\end{corollary}
\begin{proof}
From the Fourier transform of $\psi_N(x)$ we have
\begin{equation}
\widehat{\psi}_N(i\omega)=\left( \frac{e^\omega+1}{2}\right)^N.
\end{equation}
The Appell polynomials generated by binomial distributions,
$\psi_n(k):=\frac{1}{2^n}\binom{n}{k}$, are the generalized Euler
polynomials, $E_m^N(z), m=0,1,\ldots,$ that are biorthogonal to the
derivatives of $\psi_N(z)$,
\begin{equation}
\left\langle (-1)^n \psi_N^{(n)}(z), \frac{E_m^N(z)}{m!}\right\rangle=\delta _{m,n}.
\end{equation}
By Theorem (\ref{Th2}), the normalized generalized Euler polynomials
$\tilde{E}_m^N(z)$ converge to the Hermit polynomials, $H_m(x)$, as
$N\rightarrow \infty$,
\begin{equation}
\lim_{N\rightarrow\infty}\left(
\frac{4}{N}\right)^{\frac{m}{2}}E_m^N(\frac{\sqrt{N}}{2}z+\frac{N}{2})=H_m(z).
\end{equation}
\end{proof}

\begin{corollary}
Generalized Euler and Bernoulli polynomials satisfy
\begin{equation}
 B_m^N(z)=\frac{1}{2^m}\sum_{k=0}^m\binom{m}{k} E_k^N(z) B_{m-k}^N(z).
\end{equation}
\begin{proof}
Let $\widehat{\psi}_N(\omega)=\left(\frac{1+e^{-i\omega}}{2}\right)^N$ and $ \widehat{B}_N(\omega)=\left( \frac{1-e^{-i\omega}}{i\omega}\right)^N$. By (\ref{defE}) (\ref{defB}), we can see that the Euler and Bernoulli polynomials are generated by
$\widehat{\psi}_N(i\omega)$ and $\widehat{B}_N(i\omega)$ respectively. Recall that the uniform
B-spline, $B_N$, of order $N$, is the scaling function satisfying
\begin{eqnarray*}
\widehat{B}_N(\omega)&= &\widehat{\psi}_N (\omega/2)\widehat{B}_N(\omega/2),
\end{eqnarray*}
and the Fourier transform of $B_N$ is
\begin{eqnarray*}
\widehat{B}_N(\omega)=\left( \frac{1-e^{-i\omega}}{i\omega}\right)^N, && \omega\in \mathbb{R}.
\end{eqnarray*}
Therefore by Theorem \ref{2.2}, it holds
\begin{equation}
 B_m^N(z)=\frac{1}{2^m}\sum_{k=0}^m\binom{m}{k} E_k^N(z) B_{m-k}^N(z).
\end{equation}
\end{proof}
\end{corollary}

\section{Generalized Buchholz polynomials and Ultraspherical
(Gegenbauer) polynomials}
\begin{theorem}\label{Th3}
Let $\{\tilde{P}_{N,m}(x): m=0,1,\ldots\} $ be the polynomials
sequence generated by
$\frac{\tilde{f}_N(x,z)}{\widehat{\tilde{\phi}}_N(iz)}=\sum_{m=0}^\infty\frac{\tilde{P}_{N,m}(x)z^m}{m!}
$. There are constants $r>0$ and $A>0$, such that for all sufficient
large $N$, it holds
$\left|\tilde{f}_N(x,z)-e^{xz}\right|\leq\frac{A}{\sigma_N},$ and
$\left|\widehat{\tilde{\phi}}_N(iz)-e^{\frac{z^2}{2}}\right|\leq\frac{A}{\sigma_N},$
for $|z|<r.$
 Then for each $m=0,1,\ldots$, $\tilde{P}_{N,m}(x)$ converges locally uniformly to the Hermit polynomial $H_m(x)$ as $N$ goes to infinity.
\end{theorem}
\begin{remark}
When $\tilde{f}_N(x,z)=e^{xz}$, theorem (\ref{Th3}) turns to theorem
(\ref{Th2}).
\end{remark}
\begin{proof}
Since $\widehat{\tilde{\phi}}_N(0)=1$, we can choose a neighborhood
$U$ of the origin so that
$\left|\widehat{\tilde{\phi}}(iz)\right|\geq\frac{1}{2}$ and $\left|
e^{\frac{z^2}{2}}\right|\geq \frac{1}{2}$ for all $z\in U$. Take a
circle $C$ inside $U$ with center at $0$ and radius $r$, so that for
all sufficient large $N$, it holds
$\left|\tilde{f}_N(x,z)-e^{xz}\right|\leq\frac{A}{\sigma_N}$ and
$\left|\widehat{\tilde{\phi}}_N(iz)-e^{\frac{z^2}{2}}\right|\leq\frac{A}{\sigma_N},$
for $|z|<r.$

The coefficients of the Taylor series
\begin{equation}\label{taylor3}
\frac{\tilde{f}_N(x,z)
}{\widehat{\tilde{\phi}}(iz)}-\frac{e^{xz}}{e^{\frac{z^2}{2}}}=\sum_{m=0}^\infty\frac{\tilde{P}_{N,m}(x)-H_m(x)}{m!}z^m
\end{equation}
 are represented by the Cauchy's integral formula
\begin{eqnarray*}
\tilde{P}_{N,m}(x)-H_m(x)
&&=\frac{m!}{2\pi i}\oint_C\frac{1}{z^{m+1}}\left(\frac{\tilde{f}_N(x,z)}{\widehat{\tilde{\phi}}_N(iz)}-\frac{e^{xz}}{e^\frac{z^2}{2}}\right)dz\\
&&=\frac{m!}{2\pi i}\oint_C
\frac{\left(e^{\frac{z^2}{2}}\tilde{f}_N(x,z)-e^{xz}\widehat{\tilde{\phi}}(iz)\right)}{z^{m+1}\widehat{\tilde{\phi}}(iz)e^{\frac{z^2}{2}}}dz.\\
\end{eqnarray*}

Therefore, for any given real numbers $A >0$, there is a sufficient
large $N_0$, for any $N>N_0$, it holds
\begin{eqnarray*}
\left|\tilde{P}_{N,m}(x)-H_m(x)\right| &&\leq\frac{m!}{2\pi
}\oint_C\frac{\left|e^{\frac{z^2}{2}}\tilde{f}_N(x,z)-e^{xz}\widehat{\tilde{\phi}}(iz)\right|}
{\left|z^{m+1}\right|\left|\widehat{\tilde{\phi}}(iz)\right|\left|e^{\frac{z^2}{2}}\right|}|dz|\\
&&\leq\frac{m!}{2\pi
}\oint_C\frac{\left|e^{\frac{z^2}{2}}\tilde{f}_N(x,z)-e^{xz}e^{\frac{z^2}{2}}\right|+\left|e^{xz}e^{\frac{z^2}{2}}-e^{xz}\widehat{\tilde{\phi}}(iz)\right|}
{r^{m+1}\left|\widehat{\tilde{\phi}}(iz)\right|\left|e^{\frac{z^2}{2}}\right|}|dz|\\
&&\leq\frac{m!}{2\pi
}\oint_C\frac{e^{Re\frac{z^2}{2}}\left|\tilde{f}_N(x,z)-e^{xz}\right|+e^{xRez}\left|e^{\frac{z^2}{2}}-\widehat{\tilde{\phi}}(iz)\right|}
{r^{m+1}\left|\widehat{\tilde{\phi}}(iz)\right|\left|e^{\frac{z^2}{2}}\right|}|dz|\\
&&\leq\frac{4(m!)A\left(e^{\frac{r^2}{2}}+ e^{xr}\right)}{\sigma_N
r^m}
\end{eqnarray*}
Since $\sigma_N\rightarrow\infty$ as $N\rightarrow \infty$, It
follows that for each $m$, $\widehat{\tilde{P}}_{N,m(x)}\rightarrow
H_m(x)$ uniformly on compact sets.
\end{proof}

Generalized Buchholz and Ultraspherical (Gegenbauer) polynomials of
degree $m$, order $N$ and complex argument $x$, denoted respectively
by $P_m^N(x)$ and $ C_m^N(x)$, can be defined by their generating
functions,
\begin{eqnarray*}
e^{x(cotz -\frac{1}{z})/2}\left(\frac{\sin z}{z}\right)^N
=\sum_{m=0}^{\infty}P_m^N(x)z^m, && |z|<\pi\\
\end{eqnarray*}
and
\begin{eqnarray*}
(1-2xz +z^2)^{-N}=\sum_{n=0}^\infty C_m^N(x) z^m, && -1\leq x\leq1, |\omega|<1.\\
\end{eqnarray*}
Buchholz polynomials are used for the representation of the
Whittaker functions as convergent series expansions of Besse
functions \cite{Buchholz}. They appear also in the convergent
expansions of the Whittaker functions in ascending powers of their
order and in the asymptotic expansions of the Whittaker functions in
descending powers of their order \cite{L-Sesma99}. Explicit formulas
for obtaining these polynomials may be found in \cite{Abad}.

There are well known limits\cite{L_T_Bernoulli}
\begin{equation*}
\lim_{N\rightarrow\infty}\left(
\frac{3}{N}\right)^{\frac{m}{2}}P_m^N(-2\sqrt{3N}x)=\frac{1}{m!}H_m(x)
\end{equation*}
and
\begin{eqnarray*}
\lim_{N\rightarrow\infty}
(2N)^{-\frac{m}{2}}C_m^N\left(\frac{x}{\sqrt{2N}}\right)&=&\frac{1}{m!}H_m(x).
\end{eqnarray*}
These limits give insight in the location of the zeros for large
values of the limit parameter, and the asymptotic relation with the
Hermite polynomials if the parameter $N$ become large and $x$ is
properly scaled.

Many methods are available to prove these and other
limits\cite{L_T_Bernoulli}. We can get these asymptotic results from
theorem \ref{Th3} as simple cases.

\begin{lemma}\label{lemma1}
For any $|z|<\pi$,
\begin{eqnarray*}
\lim_{N\rightarrow \infty}\textrm{sinc}^N\left(\frac{z}{2}\sqrt{\frac{12}{N}}\right)=
\exp\left({-\frac{z^2}{2}}\right).
\end{eqnarray*}
\end{lemma}
\begin{proof}
Set
\begin{equation}\label{eq:2.6}
L_N(z):=N\ln
\left[\textrm{sinc}\left(\frac{z}{2}\sqrt{\frac{12}{N}}\right)\right].
\end{equation}
Then with $z_N=\frac{z}{2}\sqrt{\frac{12}{N}}$, it holds
\begin{eqnarray*}
L_N(z)&:=&N\ln
\left[\textrm{sinc}\left(\frac{z}{2}\sqrt{\frac{12}{N}}\right)\right]=3z^2\frac{\ln\left[\textrm{sinc}(\frac{z}{2}\sqrt{\frac{12}{N}})\right]}
{3z^2/N}=3z^2\frac{\ln\left[\textrm{sinc} (z_N)\right]}{z_N^2}.
\end{eqnarray*}
Since it holds $\textrm{sinc}(0)=1,\textrm{sinc}^{(1)}(0)=0$ and $\textrm{sinc}^{(2)}(0)=-\frac{1}{3}$,
for the $\lim_{N\rightarrow\infty}$, and hence $z_N\rightarrow 0$, we may apply L' H\^{o}pital's rule twice:
For any $|z|<\pi$, we have
\begin{eqnarray*}
\lim_{N\rightarrow \infty} L_N(z)&=&3z^2 \lim_{N\rightarrow \infty} \frac{\textrm{sinc}^{(1)}(z_N)}{2z_N\textrm{sinc} (z_N)}=3z^2 \lim_{N\rightarrow \infty} \frac{\textrm{sinc}^{(2)}(z_N)}{2\textrm{sinc} (z_N)+2z_N\textrm{sinc}^{(1)}(z_N)}\\
&=&3z^2\frac{\textrm{sinc}^{(2)}(0)}{2}=-\frac{z^2}{2}.
\end{eqnarray*}
It follows: For any $|z|<\pi$,
\begin{eqnarray*}
\lim_{N\rightarrow \infty}\textrm{sinc}^N\left(\frac{z}{2}\sqrt{\frac{12}{N}}\right)=
\exp\left({-\frac{z^2}{2}}\right).
\end{eqnarray*}
\end{proof}

\begin{corollary}
\begin{equation}
\lim_{N\rightarrow\infty}\left(
\frac{3}{N}\right)^{\frac{m}{2}}P_m^N(-2\sqrt{3N}x)=\frac{1}{m!}H_m(x).
\end{equation}
\end{corollary}

\begin{proof}
Let $\sigma_N=\sqrt{\frac{N}{12}}$ then
\begin{equation*}
\sum_{m=0}^\infty
\sigma_N^{-m}P_m^N(-12\sqrt{2}x\sigma_N)\left(\frac{\omega}{\sqrt{2}}\right)
^m=e^{-6\sqrt{2}x\sigma_N\left(\cot{\frac{\omega}{\sqrt{2}\sigma_N}}-\frac{\sqrt{2}\sigma_N}
{\omega}\right)}\left(\frac{\sin\frac{\omega}{\sqrt{2}\sigma_N}}
{\frac{\omega}{\sqrt{2}\sigma_N}}\right)^N
\end{equation*}
Let $\tilde{f}_N(\omega,
x)=e^{-6\sqrt{2}x\sigma_N\left(\cot{\frac{\omega}{\sqrt{2}\sigma_N}}-\frac{\sqrt{2}\sigma_N}{\omega}\right)}$
and
$\widehat{\tilde{\phi}}_N(i\omega)=\left(\frac{\sin\frac{\omega}{\sqrt{2}\sigma_N}}{\frac{\omega}{\sqrt{2}\sigma_N}}\right)^{-N}$.
By Taylor theorem, for any $|\omega|<\pi$ and sufficient large $N$,
it holds
\begin{equation*}
\ln\tilde{f}_N(\omega,x)=-6\sqrt{2}x\sigma_N\left(\cot{\frac{\omega}{\sqrt{2}\sigma_N}}-\frac{\sqrt{2}\sigma_N}{\omega}\right)=
-6\sqrt{2}x\sigma_N\left(-\frac{\omega}{3\sqrt{2}\sigma_N}+O(\sigma_N^{-3})\right).
\end{equation*}
Therefore, for $N\rightarrow +\infty$, we have
\begin{equation*}
\lim_{N\rightarrow \infty}\tilde{f}_N(\omega, z)=e^{2x\omega}.
\end{equation*}
By lemma \ref{lemma1}, we have
\begin{equation*}
\lim_{N\rightarrow
+\infty}\widehat{\tilde{\phi}}_N(i\omega)=\lim_{N\rightarrow
\infty}\left(\frac{\sin\frac{\omega}{\sqrt{2}\sigma_N}}{\frac{\omega}{\sqrt{2}\sigma_N}}\right)^{-N}
=e^{\omega^2}.
\end{equation*}
Therefore By Theorem \ref{Th3}, it holds
\begin{equation}
\lim_{N\rightarrow\infty}\left(
\frac{6}{N}\right)^{\frac{m}{2}}P_m^N(-2\sqrt{6N}x)=\frac{1}{m!}H_m(\sqrt{2}x)\sqrt{2}^m,
\end{equation}
equivalently,
\begin{equation*}
\lim_{N\rightarrow\infty}\left(
\frac{3}{N}\right)^{\frac{m}{2}}P_m^N(-2\sqrt{3N}x)=\frac{1}{m!}H_m(x).
\end{equation*}
\end{proof}

\begin{corollary}
\begin{eqnarray}
\lim_{N\rightarrow\infty}
(2N)^{-\frac{m}{2}}C_m^N\left(\frac{x}{\sqrt{2N}}\right)&=&\frac{1}{m!}H_m(x).
\end{eqnarray}
\end{corollary}
\begin{proof}
By the generating function of $C_m^N(x)$,
\begin{equation*}
(1-2xz +z^2)^{-N}=\sum_{n=0}^\infty C_m^N(x) z^m,
\end{equation*}
it holds
\begin{eqnarray*}
\sum_{n=0}^\infty
(2N)^{-\frac{m}{2}}C_m^N\left(\frac{x}{\sqrt{2N}}\right) z^m
&=&\left(1-\frac{xz}{\sqrt{N}}
+\frac{z^2}{2N}\right)^{-N}\\
&=&\left[1-\left(\frac{2xz-z^2}{2N}\right)\right]^{-\frac{2N}{2xz-z^2}(xz-z^2/2)}\\.
\end{eqnarray*}
Let $g_N=\left[1-(\frac{2xz-z^2}{2N})\right]^{-\frac{2N}{2xz-z^2}}$, then $\lim_{N\rightarrow \infty} g_N=e$.
Therefore
\begin{eqnarray*}
\sum_{n=0}^\infty (2N)^{-\frac{m}{2}}C_m^N\left(\frac{x}{\sqrt{2N}}\right) z^m&=& g_N^{xz} g_N^{-z^2/2}.
\end{eqnarray*}
and
\begin{eqnarray*}
\lim_{N\rightarrow \infty} g_N^{xz}&=& e^{xz}\\
\lim_{N\rightarrow \infty} g_N^{-z^2/2} &=&e^{-z^2/2}.
\end{eqnarray*}
By Theorem (\ref{Th3}), we have
\begin{eqnarray*}
\lim_{N\rightarrow\infty}
(2N)^{-\frac{m}{2}}C_m^N\left(\frac{x}{\sqrt{2N}}\right)&=&\frac{1}{m!}H_m(x).
\end{eqnarray*}
\end{proof}

\section{Biorthogonal systems relate to Generalized Laguerre polynomials}

Laguerre polynomials, $L_n(x)$, are solutions to the Laguerre
differential equation
\begin{eqnarray*}
xy^{\prime\prime}+(1-x)y^{\prime}+ny=0, &&n\geq 0.
\end{eqnarray*}
 Laguerre polynomials is a class of orthogonal polynomials with weighting
 function $w(x)=e^{-x}$. The Rodrigues representation for the Laguerre polynomials is
\begin{eqnarray}
L_n(x)=\frac{e^x}{n!}\frac{d^n}{dx^n}(e^{-x}x^n)
\end{eqnarray}
and the generating function for Laguerre polynomials is
\begin{eqnarray}
(1-z)^{-1} e^{\frac{-zx}{1-z}}=\sum_{m=0}^\infty L_m(x)z^m, & &
|z|\leq1.
\end{eqnarray}

The generalized Laguerre polynomials, $L^{(\alpha)}_m$, are also a
class of orthogonal polynomials with weighting function
$w(x)=x^\alpha e^{-x}$ and generated by
\begin{eqnarray}
(1-z)^{-\alpha-1} e^{\frac{-zx}{1-z}}=\sum_{m=0}^\infty
L^{(\alpha)}_m(x)z^m, & & |z|\leq 1.
\end{eqnarray}

The Rodrigues representation for the generalized Laguerre
polynomials is
\begin{eqnarray}
L^{(\alpha)}_n(x)=\frac{x^{-\alpha}e^x}{n!}\frac{d^n}{dx^n}(e^{-x}x^{n+\alpha}).
\end{eqnarray}
When $\alpha=0$, we have $L^{(0)}_n(x)=L_n(x)$.

The explicit formula for $L^{(\alpha)}_n(x)$ is\cite{xu-orthogonal}
\begin{eqnarray}
L^{(\alpha)}_n(x)=\frac{(\alpha+1)_n}{n!}\sum_{j=0}^{n}\frac{(-n)_j}{(\alpha+1)_j}\frac{x^j}{j!},
\end{eqnarray}
where $(\alpha)_n=\prod_{i=1}^n(\alpha+i-1), (\alpha)_0=1$,  for
$n=1,2,3\ldots.$ When $\alpha=0$, the explicit formula for
generalized Laguerre polynomials, $L^{(\alpha)}_n(x)$, become
\begin{eqnarray}
L_n(x)=\sum_{k=0}^n (-1)^k\binom{n}{k}\frac{x^k}{k!}.
\end{eqnarray}
By simply computing, we have
\begin{proposition}\label{L_prime}
\begin{eqnarray}
L^\prime_{n}(x)=L_{n-1}^\prime(x)-L_{n-1}(x).
\end{eqnarray}
\end{proposition}

The orthogonality relation for the Laguerre polynomials is contained
in
\begin{eqnarray}
\int_0^\infty L_m^{(\alpha)}(x)L_n^{(\alpha)}(x)x^\alpha
e^{-x}dx=\frac{\Gamma(\alpha +n+1)}{n!}\delta_{mn}, && \alpha >-1,
\end{eqnarray}
which can be considered as a biorthogonal relation between the
derivatives of $\{\frac{d^n}{dx^n}\frac{x^{\alpha+n}
e^{-x}}{\Gamma(\alpha+n+1)}\}$ and Laugerre polynomials
$\{L_m^{(\alpha)}:m=0,1,\ldots\}$,

\begin{eqnarray*}
\left\langle L_m^{(\alpha)},\frac{d^n}{dx^n}\frac{x^{\alpha+n}
e^{-x}}{\Gamma(\alpha+n+1)}\right\rangle=\delta_{m,n}.
\end{eqnarray*}

\begin{corollary}
\begin{eqnarray*}
\lim_{N\rightarrow\infty}(-1)^m
(2N)^{-m/2}L_m^{(N)}(x\sqrt{2N}+N)=\frac{1}{m!} H_m(x).
\end{eqnarray*}
\end{corollary}

\begin{proof}
\begin{eqnarray*}
\sum_{m=0}^{\infty} (-1)^m
(2N)^{-m/2}L_m^{(N)}(x\sqrt{2N}+N)z^m=\left(
1+\frac{z}{\sqrt{2N}}\right)^{-N-1}e^{\frac{\frac{\sqrt{N}z}{\sqrt{2}}}{1+z/\sqrt{2N}}}e^{\frac{zx}{1+z/\sqrt{2N}}}.
\end{eqnarray*}
Áî~$\tilde{f}_N(x,z)=e^{\frac{zx}{1+z/\sqrt{2N}}}$,
~$\widehat{\tilde{\phi}}_N(iz)=\left(1+\frac{z}{\sqrt{N}}\right)^{N+1}e^{\frac{-\frac{\sqrt{N}z}{\sqrt{2}}}{1+z/\sqrt{2N}}}$,
then for ~$N\rightarrow\infty$, we have
\begin{eqnarray*}
&\lim_{N\rightarrow\infty}\tilde{f}_N(x,z)=e^{zx},\\
&\lim_{N\rightarrow\infty}\widehat{\tilde{\phi}}_N(iz)=e^{\frac{z^2}{2}}.
\end{eqnarray*}
By Theorem\ref{Th3}, we have
\begin{eqnarray*}
\lim_{N\rightarrow\infty}(-1)^m
(2N)^{-m/2}L_m^{(N)}(x\sqrt{2N}+N)=\frac{1}{m!} H_m(x).
\end{eqnarray*}
\end{proof}

In this section, we consider a family of biorthogonal polynomials,
$\{P_m(x,\alpha,\omega):m=0,1,\ldots\}$, generated by a sequence of
functions, $\phi(x,\alpha,\omega)$, which converges to the
generalized Laguerre polynomials, $L_{m}^{(\alpha)}(x)$.

Taking a compactly supported distribution $\phi\in\mathcal
{E}^\prime(\mathbb{R})$, Let $\hat{\phi}$ denote the Laplace
transform of $\phi$. Then for any integer $n\geq0$,
\begin{equation}
\left\langle \phi^{(n)}, (1-z)^n
e^{\frac{-zx}{1-z}}\right\rangle=\left\langle \phi(x), z^n
e^{\frac{-zx}{1-z}}\right \rangle =z^n\widehat{\phi}(\frac{z}{1-z})
\end{equation}
If $\widehat{\phi}(0)\neq0$,
\begin{equation}\label{Z2}
\left\langle \phi^{(n)}, \frac{(1-z)^n
e^{\frac{-zx}{1-z}}}{\widehat{\phi}(\frac{z}{1-z})}\right\rangle=z^n
\end{equation}
in a neighborhood of $0$. Since $\phi$ is compactly supported, $\widehat{\phi}$ is analytic. So we can
define a sequence of polynomials, $P_m$, by the generating function
\begin{equation}\label{L_m}
\frac{(1-z)^n
e^{\frac{-zx}{1-z}}}{\widehat{\phi}(\frac{z}{1-z})}=\sum_{m=0}^\infty
P_m(x)z^m.
\end{equation}
It follows from (\ref{Z2}) and (\ref{L_m}) that for any integer $n\geq 0$,
\begin{equation}
z^n=\sum_{m=0}^\infty \left\langle\phi^{(n)},P_m(x)\right\rangle
z^m,
\end{equation}
which gives the biorthogonal relation
\begin{equation}\label{La_bio}
\left\langle\phi^{(n)},P_m(x)\right\rangle=\delta_{m,n}.
\end{equation}
Differentiating (\ref{L_m}) with respect to $x$ and equating coefficients of $z^m$ in the resulting equation gives
\begin{eqnarray}\label{Lre}
P'_m(x)=P'_{m-1}(x)-P_{m-1}(x), && m=1,2,\ldots,
\end{eqnarray}
which is similar to the property of Laguerre polynomials
(Proposition \ref{L_prime}),
\begin{equation*}
L'_m(x)=L'_{m-1}(x)-L_{m-1}(x).
\end{equation*}

If $\phi(x,\alpha)=\frac{x^{\alpha+n}e^{-x}}{\Gamma(\alpha+n+1)}$,
the Laplace transform of $\phi(x,\alpha)$ is
$\hat{\phi}(z,\alpha)=\frac{1}{(z+1)^{\alpha+n+1}}$, which implies
$\hat{\phi}(\frac{z}{1-z},\alpha)=(1-z)^{\alpha+n+1}$. We have
\begin{eqnarray}
\left\langle \phi^{(n)}(x,\alpha),
\frac{e^{-\frac{z}{1-z}}(1-z)^n}{\hat{\phi}(\frac{z}{1-z},\alpha)}
\right\rangle=z^n.
\end{eqnarray}
in a neighborhood of 0. We can define a sequence of polynomials,
$L_m^{(\alpha)}(x)$, by the generating function
\begin{eqnarray*}
\frac{e^{-\frac{z}{1-z}}(1-z)^n}{\hat{\phi}(\frac{z}{1-z},\alpha)}=(1-z)^{-\alpha-1}e^{-\frac{z}{1-z}}=\sum_{m=0}^\infty
L_m^{(\alpha)}(x)z^m.
\end{eqnarray*}
So the biorthogonal systems generated by function
$\phi(x,\alpha)=\frac{x^{\alpha+n}e^{-x}}{\Gamma(\alpha+n+1)}$ are
Laguerre polynomials.

\begin{theorem}\label{th4} Let $\phi(x,\alpha,\omega) $ satisfy the following
conditions:

(1) There are constants $0<r<1$ and $c$, for any $\varepsilon>0$,
$\exists \delta$, such that $ |\omega-c|<\delta$, it holds
\begin{eqnarray}\label{abs}
\left|\widehat{ \phi}(\frac{z}{1-z},\alpha,\omega)-(1-z)^{\alpha
+n+1}\right|\leq \varepsilon, && \left|z\right|<r.
\end{eqnarray}
Equivalently, $\lim_{\omega\rightarrow
c}\phi(x,\alpha,\omega)=\frac{e^{-x}x^{n+\alpha}}{\Gamma(\alpha+n+1)}$.

(2) Let $\{P_{m}(x,\alpha,\omega): m=0,1,\ldots\} $ be the
biorthogonal polynomials generated by the functions,
$\phi(z,\alpha,\omega)$ by

\begin{eqnarray*}
\frac{e^{-\frac{z}{1-z}}(1-z)^n}{\hat{\phi}(\frac{z}{1-z},\alpha,\omega)}=\sum_{m=0}^\infty
P_m(x,\alpha,\omega)z^m.
\end{eqnarray*}
 Then for each $m=0,1,\ldots$,
$\{P_{m}(x,\alpha,\omega): m=0,1,\ldots\}$ converges locally
uniformly to the generalized Laguerre polynomial,
$L^{(\alpha)}_m(x)$, as $\omega$ goes to $c$.
\end{theorem}

\begin{proof}
Since $\widehat{\phi}(0,\alpha,\omega)=1$, we can choose a
neighborhood $U$ of the origin so that
$\left|\widehat{\phi}(\frac{z}{1-z},\alpha,\omega)\right|\geq\frac{1}{2}$
and $\left| (1-z)^{\alpha+1}\right|\geq \frac{1}{2}$ for all $z\in
U$. Take a circle $C$ completely contained in $U$, with centra at
the origin $0$ and radius $r$, so that ({\ref{abs}}) is satisfied.
The coefficients of the Taylor series
\begin{equation}\label{Taylor2}
\frac{(1-z)^n
e^{\frac{-zx}{1-z}}}{\widehat{\phi}(\frac{z}{1-z},\alpha,\omega)}-\frac{e^{\frac{-zx}{1-z}}}{(1-z)^{\alpha+1}}
=\sum_{m=0}^\infty\left(P_{m}(x,\alpha,\omega)-L_m^{(\alpha)}(x)\right)z^m
\end{equation}
are represented by the Cauchy's integral formula:
\begin{equation*}
P_{m}(x,\alpha,\omega)-L_m^{(\alpha)}(x)=\frac{1}{2\pi
i}\oint_Ce^{\frac{-zx}{1-z}}\frac{(1-z)^{\alpha+n+1}
-\widehat{\phi}(\frac{z}{1-z},\alpha,\omega)}{z^{m+1}\widehat{\phi}(\frac{z}{1-z},\alpha,\omega)(1-z)^{\alpha+1}}dz
\end{equation*}
\begin{eqnarray*}
\left|P_{m}(x,\alpha,\omega)-L_m^{(\alpha)}(x)\right|&&\leq\frac{1}{2\pi
}\oint_C
\left|e^{\frac{-zx}{1-z}}\right|\frac{\left|(1-z)^{\alpha+n+1}
-\widehat{\phi}(\frac{z}{1-z},\alpha,\omega)\right|}{r^{m+1}\left|\widehat{\phi}(\frac{z}{1-z},\alpha,\omega)\right|\left|(1-z)^{\alpha+1}\right|}dz\\
&&\leq\frac{1}{2\pi }\oint_C \frac{e^{xRe(\frac{-z}{1-z})}
\varepsilon}{r^{m+1}\left|\widehat{\phi}(\frac{z}{1-z},\alpha,\omega)\right|\left|(1-z)^{\alpha+1}\right|}dz\\
&&\leq4e^{xRe(\frac{-z}{1-z})}\varepsilon.\\
\end{eqnarray*}
It follows that for each $m$,  $P_{m}(x,\alpha,\omega)$ converges
locally uniformly to the generalized Laguerre polynomial,
$L^{(\alpha)}_m(x)$, as $\omega$ goes to $c$.
\end{proof}

For the Meixner-Pollaczek polynomials, we have the generating function:
\begin{equation}
F(x,z)=\left( 1-e^{i\omega}z\right)^{-\lambda+ix}\left(
1-e^{-i\omega}z\right)^{-\lambda-ix}=\sum_{n=0}^\infty
P_m^{(\lambda)}(x;\omega)z^m.
\end{equation}
\begin{corollary}
The Laguerre polynomials can be obtained from Meixner-Pllaczek
polynomials by the substitution $\lambda=\frac{1}{2}(\alpha +1),
x\rightarrow -\frac{1}{2}\omega^{-1}x$ and letting
$\omega\rightarrow 0$.
\begin{eqnarray*}
\lim_{\omega\rightarrow 0} P_n^{\frac{\alpha+1}{2}}(\frac{-x}{2\omega};\omega)=L_n^{(\alpha)}(x).
\end{eqnarray*}
\end{corollary}
\begin{proof}
\begin{eqnarray*}
&&\sum_{m=0}^\infty P_n^{(\frac{\alpha+1}{2})}
\left(\frac{-x}{2\omega}; \omega\right)z^m \\
&&=\left[ \left(
1-ze^{i\omega} \right) \left( 1-ze^{-i\omega}\right)\right]^{-\frac
{\alpha+1}{2}}\left( \frac{1-ze^{i\omega}}{1-ze^{-i\omega}}\right)^{-\frac{ix}{2\omega}}\\
\end{eqnarray*}
Let $\widehat{\phi}(\frac{z}{1-z},\alpha,\omega)=\left[ \left(
1-ze^{i\omega} \right)\left(
1-ze^{-i\omega}\right)\right]^{\frac{\alpha+1}{2}}(1-z)^n$, it holds
\begin{eqnarray*}
\lim_{\omega\rightarrow
0}\widehat{\phi}(\frac{z}{1-z},\alpha,\omega)&=&\lim_{\omega\rightarrow
0} \left[ \left( 1-ze^{i\omega} \right)\left(
1-ze^{-i\omega}\right)\right]^{\frac{\alpha+1}{2}}
(1-z)^n\\
&=&\lim_{\omega\rightarrow 0}\left[ 1-2z\cos\omega +z^2\right]^{-\frac{\alpha+1}{2}}(1-z)^n\\
&=&(1-z)^{\alpha+n+1}.\\
\end{eqnarray*}

Since
\begin{eqnarray*}
\lim_{\omega\rightarrow 0}\left(
\frac{1-ze^{i\omega}}{1-ze^{-i\omega}}\right)^{-\frac{ix}{2\omega}}
&=&\lim_{\omega\rightarrow
0}\left(1+z\frac{e^{-i\omega}-e^{i\omega}}{1-ze^{-i\omega}}
\right)^{-\frac{ix}{2\omega}}\\
&=&\lim_{\omega\rightarrow 0}\left(1+z\frac{e^{-i\omega}-e^{i\omega}}{1-ze^{-i\omega}} \right)^{\frac{1-e^{-i\omega}}{-2iz\sin\omega} \cdot\frac{zx\sin\omega}{\left(1-ze^{-i\omega}\right)\omega}}\\
&=&e^{\frac{zx}{1-z}},
\end{eqnarray*}
by theorem (\ref{th4}), we have
\begin{eqnarray*}
\lim_{\omega\rightarrow 0} P_n^{\frac{\alpha+1}{2}}(\frac{-x}{2\omega};\omega)=L_n^{(\alpha)}(x).
\end{eqnarray*}
\end{proof}

The generating function for Meixner polynomials $M_n(x;\beta,c)$ is
\begin{eqnarray*}
\left( 1-\frac{z}{c}\right)^x(1-z)^{-\beta-x}=\sum_{n=0}^\infty
\frac{(\beta)_n}{n!}M_n(x;\beta,c)z^n.
\end{eqnarray*}

\begin{corollary}
The Laguerre polynomials can be obtained from Meixner polynomials by
the substitution $\beta=\alpha +1, x\rightarrow \frac{cx}{1-c}$ and
letting $c\rightarrow 1$.
\begin{eqnarray*}
\lim_{c\rightarrow 1} M_n\left(\frac{cx}{1-c};
\alpha+1,c\right)=\frac{L_n^{(\alpha)}(x)}{L_n^{(\alpha)}(0)}.
\end{eqnarray*}
\end{corollary}
\begin{proof}
\begin{eqnarray*}
&&\sum_{n=0}^\infty \frac{(\alpha+1)_n}{n!} M_n\left(\frac{cx}{1-c};
\alpha+1,c\right)z^n\\
&&=\left(1-\frac{z}{c}
\right)^{\frac{cx}{1-c}}\left(1-z\right)^{-\alpha-1-\frac{cx}{1-c}}\\
&&=\left( \frac{1-\frac{z}{c}}{1-z}\right)^{\frac{cx}{1-c}}(1-z)^{-\alpha-1}\\
&&=\left(1+\frac{z(c-1)}{c(1-z)}\right)^{\frac{cx}{1-c}}(1-z)^{-\alpha-1}\\
&&=\left(1+\frac{z(c-1)}{c(1-z)}\right)^{\frac{(1-z)c}{z(c-1)}\frac{-z x}{1-z}}(1-z)^{-\alpha-1}\\
\end{eqnarray*}
Since
\begin{equation}
\lim_{c\rightarrow
1}\left(1+\frac{z(c-1)}{c(1-z)}\right)^{\frac{(1-z)c}{z(c-1)}\frac{-z
x}{1-z}}=e^{-\frac{z}{1-z}x}
\end{equation}
Therefore, by theorem(\ref{th4}), it holds
\begin{eqnarray*}
\lim_{c\rightarrow1}(\alpha+1)_n M_n\left( \frac{cx}{1-c};
\alpha+1,c\right)=L_n^{(\alpha)}(x).
\end{eqnarray*}
Since
\begin{eqnarray*}
L_n^{(\alpha)}(x)=\frac{(\alpha+1)_n}{n!} \sum_{j=0}^n \frac{(-n)_jx^j}{(\alpha+1)_jj!},
\end{eqnarray*}
then $L_n^\alpha(0)=(\alpha+1)_n$.
Therefore
\begin{eqnarray*}
\lim_{c\rightarrow 1} M_n\left(\frac{cx}{1-c};
\alpha+1,c\right)=\frac{L_n^{(\alpha)}(x)}{L_n^{(\alpha)}(0)}.
\end{eqnarray*}
\end{proof}

\end{document}